\documentclass[11pt]{article}
\usepackage{hyperref}
\usepackage[T1]{fontenc}
\usepackage[utf8]{inputenc}
\usepackage{amsmath,amssymb,amsthm,mathtools}
\usepackage[a4paper,margin=20mm]{geometry}
\usepackage{enumerate}
\newtheorem{theorem}{Theorem}[section]

\usepackage{url}

\title{Lack of exponential stability for thermoelastic Timoshenko and Bresse systems with full homogeneous Dirichlet boundary conditions}
\author{Pedro Roberto de Lima\\
{\small \textit{Federal Rural University of Rio de Janeiro, Seropédica, Rio de Janeiro, Brazil}}}
\date{}

\begin{document}
\maketitle

\begin{abstract}
In this paper we prove that the classical thermoelastic Timoshenko and Bresse systems with thermal coupling acting only on the bending moment are not exponentially stable under full homogeneous Dirichlet boundary conditions if the equal wave speed condition is not satisfied. Combined with their well-known converses, our results provide complete characterizations of exponential stability in terms of the coefficients of each system. To the best of our knowledge, these characterizations were previously known only under mixed boundary conditions. The proof is based on the introduction of a boundary correction term in the standard trigonometric ansatz so that the full homogeneous Dirichlet conditions are satisfied.
\end{abstract}

\section{Introduction}

The classical linear Timoshenko system with thermal effects acting only on the bending moment is given by
\begin{equation}\label{sistema}\left\{\begin{aligned}
&\rho_1\varphi_{tt}-k(\varphi_x+\psi)_x=0,\\
&\rho_2\psi_{tt}-b\psi_{xx}+k(\varphi_x+\psi)+\gamma\theta_x=0,\\
&\rho_3\theta_t-\kappa\theta_{xx}+\gamma\psi_{xt}=0,
\end{aligned}\right.\end{equation}
where $\varphi$, $\psi$ and $\theta$ are functions of time $t\geq 0$ and position $x\in[0,\ell]$ describing, respectively, the vertical displacement, the angle of rotation of the cross section and the temperature deviation of a planar and initially straight beam of length $\ell$, and all coefficients are positive constants. This system was first studied in the influential work \cite{Rivera02}, where the authors proved (among other things) the following results:
\begin{itemize}
    \item For the mixed boundary conditions
    \begin{equation}\label{bcm}
\varphi=\psi_x=\theta=0\quad\text{on}\quad[0,\infty)\times \{0,\ell\},
\end{equation}
    system \eqref{sistema} is exponentially stable if and only if the equal wave speed condition
    \begin{equation}\label{EWSC}
    \frac{\rho_1}{k}=\frac{\rho_2}{b}
    \end{equation}
    is satisfied.
     \item For the full homogeneous Dirichlet boundary conditions
    \begin{equation}\label{bcd}
\varphi=\psi=\theta=0\quad\text{on}\quad[0,\infty)\times \{0,\ell\},
\end{equation}
    system \eqref{sistema} is exponentially stable if \eqref{EWSC} is satisfied.
\end{itemize}

In \cite{Fatori10}, these results were extended to the Bresse version of \eqref{sistema}, given by
\begin{equation}\left\{\begin{aligned}\label{sistemab}
&\rho_{1}\varphi_{tt} - k(\varphi_{x}+\psi+lw)_{x} - k_{0}l(w_{x}-l\varphi)=0,\\
&\rho_{2}\psi_{tt} - b\psi_{xx}+k(\varphi_{x} + \psi+lw)+\gamma\theta_x=0,\\
&\rho_{1}w_{tt} - k_{0}(w_{x}-l\varphi)_{x} + kl(\varphi_{x}+\psi+lw)=0,\\
&\rho_3\theta_t-\kappa\theta_{xx}+\gamma\psi_{xt}=0,
\end{aligned}\right.\end{equation}
where $w$ is the longitudinal displacement, also a function of $t$ and $x$, and $l$ and $k_0$ are also positive constants. In this case:
 \begin{itemize}
    \item For the mixed boundary conditions
$$
\varphi=\psi_x=w_x=\theta=0\quad\text{on}\quad[0,\infty)\times \{0,\ell\},
$$
    system \eqref{sistemab} is exponentially stable if and only if the more general equal wave speed condition
    \begin{equation}\label{EWSCb}
    \frac{\rho_1}{k}=\frac{\rho_2}{b}\quad \text{and}\quad k=k_0
    \end{equation}
    is satisfied.
     \item For the full homogeneous Dirichlet boundary conditions
    \begin{equation}\label{bcdb}
\varphi=\psi=w=\theta=0\quad\text{on}\quad[0,\infty)\times \{0,\ell\},
\end{equation}
    system \eqref{sistemab} is exponentially stable if \eqref{EWSCb} is satisfied.
\end{itemize}

In these works, the questions of whether condition \eqref{EWSC} is also necessary for the exponential stability of \eqref{sistema} under the boundary conditions \eqref{bcd}, and whether condition \eqref{EWSCb} is also necessary for the exponential stability of \eqref{sistemab} under the boundary conditions \eqref{bcdb}, were left open.

The result on the lack of exponential stability of \eqref{sistema}-\eqref{bcm} has also been extended to many other variants and generalizations of \eqref{sistema}. For example, in \cite{Santos12}, it was extended to the Cattaneo version of \eqref{sistema}; in \cite{Dilberto14}, it was extended to the case where the thermal coupling acts only on the shear force; in \cite{Cui23}, it was extended to a version of \eqref{sistema} with Kelvin–Voigt damping. What all these extensions (and many others) have in common is that they are based on a trigonometric ansatz for the solution of the resolvent equation associated with the system. This ansatz works for suitable mixed boundary conditions but is incompatible with full homogeneous Dirichlet boundary conditions. The difficulty of dealing with this type of boundary conditions has been recognized in the literature. For example:
\begin{itemize}
    \item In the cited work \cite{Fatori10}, the authors stated that, although the lack of exponential stability should be true in the case of full Dirichlet conditions, they were unaware of a formal proof.
    \item In \cite{Liu09, Abdallah18}, which considered different versions of \eqref{sistemab}, the authors mentioned that the approach that uses the trigonometric ansatz to prove the lack of exponential stability had not been successful in the literature under full Dirichlet boundary conditions. In particular, the authors of \cite{Abdallah18} used a different approach to deal with their system.
\end{itemize}

    It is also worth mentioning that, in the work \cite{Cardozo19}, which provides a comprehensive analysis of system \eqref{sistema} with different boundary conditions, the authors used the usual trigonometric ansatz to prove optimality of the polynomial decay rate and left the question of optimality open in the case of full Dirichlet conditions. Since the lack of exponential stability was obtained as a consequence of this optimality, it was not established in their work for \eqref{sistema} with boundary conditions \eqref{bcd}.

Therefore, as far as we know, the following questions have not yet been answered:
\begin{enumerate}[(a)]
    \item\label{a} Do different wave speeds imply the lack of exponential stability for \eqref{sistema} with boundary conditions \eqref{bcd}?
    \item\label{b} Do different wave speeds imply the lack of exponential stability for \eqref{sistemab} with boundary conditions \eqref{bcdb}?
\end{enumerate}
In this work, we show that the implication in \eqref{a} can be proved by introducing a boundary correction term in the usual trigonometric ansatz that works for \eqref{bcm} (see Theorem \ref{mainthm}). We also show that an adaptation of the argument can be used to prove the implication in \eqref{b} (see Theorem \ref{mainthmb}). This work suggests a possible way to extend the standard trigonometric ansatz approach from mixed to full homogeneous Dirichlet boundary conditions in other variants of Timoshenko and Bresse systems that have been studied by many authors, complementing a part of the analysis that is generally left open.

\section{Methodology}

A standard approach for proving lack of exponential stability for Timoshenko and Bresse systems is based on the following well-known result.

\begin{theorem}[Gearhart-Pr\"{u}ss Theorem, see \cite{Liu99}]\label{2506262} Let $\mathcal{H}$ be a complex Hilbert space and $\mathcal{A}:D(\mathcal{A})\subset \mathcal{H}\to\mathcal{H}$ the infinitesimal generator of a $C_0$-semigroup of contractions on $\mathcal{H}$. The semigroup generated by $\mathcal{A}$ is exponentially stable if and only if the following two conditions are satisfied:
\begin{enumerate}[(i)]
    \item\label{2506263} $\mathbf{i}\mathbb R\subset \rho(\mathcal{A})$.
    \item\label{2506261} $\displaystyle\underset{{|\beta|\to\infty}}{\limsup}\|(\mathbf{i}\beta I-\mathcal{A})^{-1}\|<\infty$.
\end{enumerate}
\end{theorem}

Usually, given that condition \eqref{2506263} has already been established, the lack of exponential stability is attained by proving that condition \eqref{2506261} is not satisfied. For this, it is enough to show that there exist sequences $(\beta_n)\subset\mathbb{R}$ and $(U_n)\subset D(\mathcal{A})$ such that
\begin{equation}\label{1209261}
\beta_n\to\infty,\qquad \text{$((\mathbf{i}\beta_nI-\mathcal{A})U_n)$ is bounded in $\mathcal{H}$},\qquad 
\|U_n\|_{\mathcal{H}}
\to\infty
\end{equation}
because, in this case, if we write $F_n=(\mathbf{i}\beta_nI-\mathcal{A})U_n$, we conclude that there exists a constant $C>0$ such that
$$\|U_n\|_{\mathcal{H}}=\|(\mathbf{i}\beta_n I-\mathcal{A})^{-1}F_n\|_{\mathcal{H}}
\leq C \|(\mathbf{i}\beta_n I-\mathcal{A})^{-1}\|_{\mathcal{L}}$$
and, as a result, $$\begin{aligned}\underset{{|\beta|\to\infty}}{\limsup}\|(\mathbf{i}\beta I-\mathcal{A})^{-1}\|_{\mathcal{L}}&\geq \underset{{n\to\infty}}{\limsup}\|(\mathbf{i}\beta_n I-\mathcal{A})^{-1}\|_{\mathcal{L}}\geq \underset{{n\to\infty}}{\limsup}\, C^{-1}\|U_n\|_{\mathcal{H}}=\infty.
\end{aligned}$$

In the case of mixed boundary conditions, we usually first choose a bounded sequence $(F_n)\subset \mathcal{H}$ and then compute the coefficients of a suitable ansatz $U_n\in D(\mathcal{A})$, which solves $(\mathbf{i}\beta_n I-\mathcal{A})U_n=F_n$,  $(F_n)$ being chosen so that the computed coefficients give $\|U_n\|_{\mathcal{H}}\to\infty$. In this case, $F_n$ determines the coefficients of the ansatz. In the approach presented below for the full homogeneous Dirichlet boundary conditions, we first define $(U_n)\subset D(\mathcal{A})$ by introducing a boundary correction term in the ansatz and then choose the coefficients so that the sequence $F_n:=(\mathbf{i}\beta_n I-\mathcal{A})U_n$ remains bounded in $\mathcal{H}$ while $\|U_n\|_{\mathcal{H}}\to\infty$. In this case, the coefficients of the ansatz determine $F_n$. 

This approach was developed by the author through interaction with a large language model. More precisely, the author used ChatGPT-Astra (accessed September 2026) to assist in developing the proof strategy and the construction of the needed sequences. The author formulated the research questions, directed the interactions, critically evaluated the proposed constructions, and independently verified all mathematical arguments and calculations. The author adapted and simplified the proposed constructions, wrote the manuscript, made all final decisions regarding the mathematical content and presentation and takes full responsibility for the results and conclusions.

\section{Timoshenko system}

\subsection{Background}

Let $L^2=L^2(0,\ell)$, $H^2=H^2(0,\ell)$ and $H_0^1=H_0^1(0,\ell)$ be the usual Lebesgue and Sobolev spaces. Write $U=(\varphi,\Phi,\psi,\Psi,\theta)$, $\tilde{U}=(\tilde{\varphi},\tilde{\Phi},\tilde{\psi},\tilde{\Psi},\tilde{\theta})$ and define
$$\mathcal{H}=[H_0^1\times L^2]^2\times L^2
$$
with norm
\begin{equation}\label{norm}
    \begin{aligned}
\|U\|_{\mathcal {H}}^2=&\;\rho_1\|\Phi\|_{2}^2+\rho_2\|\Psi\|_{2}^2+k\|\varphi_x+\psi\|_{2}^2+b\|\psi_x\|_{2}^2+{\rho_3}\|\theta\|_{2}^2
\end{aligned}
\end{equation}
induced by the inner product
$$\begin{aligned}
(\tilde{U},U)_{\mathcal {H}}=&\;\rho_1(\tilde{\Phi},\Phi)_{2}+\rho_2(\tilde{\Psi},\Psi)_{2}+k(\tilde{\varphi}_x+\tilde{\psi},\varphi_x+\psi)_{2}+b(\tilde{\psi}_x,\psi_x)_{2}+{\rho_3}(\tilde{\theta},\theta)_{2}.
\end{aligned}$$

Now, define $\mathcal{A}:D(\mathcal{A})\to \mathcal{H}$
by
$$\mathcal{A}U=\begin{bmatrix}
\Phi\\
\tfrac{1}{\rho_1}k(\varphi_x+\psi)_x\\
\Psi\\
\tfrac{1}{\rho_2}[b\psi_{xx}-k(\varphi_x+\psi)-\gamma\theta_x]\\
\frac{1}{{\rho_3}}[\kappa\theta_{xx}-\gamma\Psi_x]
\end{bmatrix}
$$
with domain
$$ D(\mathcal{A})=[(H^2\cap H_0^1)\times H_0^1]^2\times (H^2\cap H^1_0).
$$

In terms of this operator, system \eqref{sistema} with boundary conditions \eqref{bcd} and initial conditions
$$
\varphi=\varphi_0,\quad \varphi_t=\varphi_1,\quad \psi=\psi_0,\quad \psi_t=\psi_1,\quad \theta=\theta_0 
\qquad\text{on}\qquad\{0\}\times (0,\ell)
$$
can be written as
\begin{equation}\label{PAC}\left\{\begin{aligned}
&U_t(t)=\mathcal{A} U(t),\quad t> 0\\
&U(0)=U_0,
\end{aligned}\right.\end{equation}
where $U(t)=(\varphi(t),\allowbreak\varphi_t(t),\allowbreak\psi(t),\allowbreak\psi_t(t),\allowbreak \theta(t))$ and $U_0=(\varphi_0,\allowbreak \varphi_1,\allowbreak \psi_0,\allowbreak \psi_1,\allowbreak \theta_0)$.

\begin{theorem}[see \cite{Rivera02,Cardozo19}] Let $\mathcal{A}$ be the operator defined above. Then:
\begin{enumerate}[(a)]
    \item $\Re(\mathcal{A}U,U)_{\mathcal{H}}=-\kappa\|\theta_x\|_2^2$ for all $U\in D(\mathcal{A})$.
    \item $\mathbf{i}\mathbb R\subset \rho(\mathcal{A})$.
    \item $\mathcal{A}$ is the infinitesimal generator of a $C_0$ semigroup of contractions on $\mathcal{H}$ and thus \eqref{PAC} has a unique classical solution for each $U_0\in D(\mathcal{A})$, and a unique mild solution for each $U_0\in \mathcal{H}$.
    \item\label{d} If condition \eqref{EWSC} is satisfied, then the semigroup generated by $\mathcal{A}$ is exponentially stable.
\end{enumerate}
\end{theorem}

\subsection{Main result}

In this section, we prove the converse of item \eqref{d}, obtaining a complete characterization of the exponential stability in terms of the coefficients of the system: system \eqref{sistema} with Dirichlet boundary conditions \eqref{bcd} is exponentially stable if and only if ${\rho_1}/{k}={\rho_2}/{b}$. To the best of our knowledge, this result was previously known only for mixed boundary conditions, while the case of full Dirichlet boundary conditions has usually remained open in the context of Timoshenko systems.

\begin{theorem}\label{mainthm} Let $\mathcal{A}$ be the operator defined in the previous section. If 
\begin{equation}\label{nEWSC}
\frac{\rho_1}{k}\neq \frac{\rho_2}{b},
\end{equation}
then the semigroup generated by $\mathcal{A}$ is not exponentially stable.
    
\end{theorem}

\begin{proof} It is enough to show that there exist sequences $(\beta_n)\subset\mathbb{R}$ and $(U_n)\subset D(\mathcal{A})$ satisfying \eqref{1209261}. In order to construct such sequences, choose a real function ${v}\in C^2[0,\infty)$
such that ${v}(0)=1$ and ${v}(s)=0$ for all $s\geq 1$. For each positive integer $n$, define
\begin{equation}\label{1309266}
c_n=\frac{n\pi}{\ell}
\end{equation}
and
$$
r_n(x)=\underbrace{{v}(c_nx)}_{:=a_n(x)}+\underbrace{(-1)^n{v}(c_n(\ell-x))}_{:=b_n(x)},
\quad x\in[0,\ell].
$$

Since $c_n\ell=n\pi>1$, we conclude that ${v}(c_n\ell)=0$. Therefore,
\begin{equation}\label{1209262}
r_n(0)=1\quad\text{and}\quad r_n(\ell)=(-1)^n.
\end{equation}

Using the changes of variable $s=c_nx$ and $s=c_n(\ell-x)$, we obtain
\begin{equation}\label{1709262}
\begin{aligned}
\|r_n\|_2^2\leq C(\|a_n\|_2^2+\|b_n\|_2^2)&= C\left( \frac{1}{c_n}\int_0^{n\pi} |{v}(s)|^2\,ds+\frac{1}{c_n}\int_0^{n\pi} |{v}(s)|^2\,ds\right)\\
&=\frac{2C\ell}{n\pi}\int_0^1|{v}(s)|^2\,ds=O(n^{-1}).
\end{aligned}
\end{equation}

As
$$ r_n' (x)=c_n{v}'(c_nx)+c_n(-1)^{n+1}{v}'(c_n(\ell-x))$$
and
$$ r_n'' (x)=c_n^2{v}''(c_nx)+c_n^2(-1)^{n+2}{v}''(c_n(\ell-x)),$$
an analogous reasoning shows that
\begin{equation}\label{1709263}
\| r_n' \|_2^2\leq C\left(c_n\int_0^{n\pi} |{v}'(s)|^2\,ds+c_n\int_0^{n\pi} |{v}'(s)|^2\,ds\right)=O(n)
\end{equation}
and
\begin{equation}\label{1709264}
\| r_n'' \|_2^2\leq C\left(c_n^3\int_0^{n\pi} |{v}''(s)|^2\,ds+c_n^3\int_0^{n\pi} |{v}''(s)|^2\,ds\right) = O(n^3).
\end{equation}

Now, let us take $\beta_n\in\mathbb R$ and $A_n,B_n,D_n\in\mathbb C$ to be chosen later and define $U_n=
(\varphi_n,\allowbreak\Phi_n,\allowbreak\psi_n,\allowbreak\Psi_n,\allowbreak\theta_n)$, where
\begin{equation}\label{1309263}
\begin{array}{lll}
\varphi_{n}(x)=A_n\sin (c_nx), &\psi_{n}(x)
=B_n\bigl[\cos(c_nx)-r_n(x)\bigr],\\
\Phi_{n}(x)
=\mathbf{i}\beta_nA_n\sin (c_nx), &
\Psi_{n}(x)
=\mathbf{i}\beta_nB_n
\bigl[\cos(c_nx)-r_n(x)\bigr],\\
\theta_{n}(x)
=D_n\sin (c_nx).&
\end{array}
\end{equation}
Due to \eqref{1209262}, it is clear that $\varphi_n,\Phi_n,\psi_n,\Psi_n,\theta_n\in H_0^1$ and thus $U_n\in D(\mathcal{A})$. Let us write
\begin{equation}\label{1309262}
(\mathbf{i}\beta_n I-\mathcal{A})U_n=(f_{1n},f_{2n}, f_{3n},f_{4n},f_{5n}).    
\end{equation}
Then,
\begin{equation}\label{1709265}
\begin{aligned}
\|(\mathbf{i}\beta_n I-\mathcal{A})U_n\|^2_{\mathcal{H}}\overset{\eqref{norm}}{=}&\,
\rho_1\|f_{2n}\|_{2}^2+\rho_2\|f_{4n}\|_{2}^2+k\|f_{1n,x}+f_{3n}\|_{2}^2+b\|f_{3n,x}\|_{2}^2\\
&  +{\rho_3}\|f_{5n}\|_{2}^2.
\end{aligned}
\end{equation}
In terms of its components, \eqref{1309262} can be written as
\begin{equation}\label{1709261}
\begin{aligned}
\mathbf{i}\beta_n\varphi_{n}-\Phi_{n}
&=f_{1n},\\
\rho_1\mathbf{i}\beta_n\Phi_{n}
-k(\varphi_{n,x}+\psi_{n})_x
&=\rho_1f_{2n},\\
\mathbf{i}\beta_n\psi_{n}-\Psi_{n}
&= f_{3n},\\
\rho_2\mathbf{i}\beta_n\Psi_{n}
-b\psi_{n,xx}
+k(\varphi_{n,x}+\psi_{n})
+\gamma\theta_{n,x}
&=\rho_2f_{4n},\\
\rho_3\mathbf{i}\beta_n\theta_{n}
-\kappa\theta_{n,xx}
+\gamma\Psi_{n,x}
&=\rho_3f_{5n}.
\end{aligned}
\end{equation}

Substituting \eqref{1309263} in \eqref{1709261}, we obtain
\begin{align}
f_{1n}&=0,\label{f1}\\
\rho_1f_{2n}
&=
\bigl[{\underbrace{(kc_n^2-\rho_1\beta_n^2)A_n}_{:=I_1}}+kc_nB_n\bigr]
\sin (c_nx)
+kB_n r_n' ,\label{f2}\\
 f_{3n}&=0,\label{f3}\\
\begin{split}
\rho_2f_{4n}
&=
\bigl[{\underbrace{kc_nA_n+(bc_n^2-\rho_2\beta_n^2)B_n}_{:=I_2}}+kB_n
+\gamma c_nD_n\bigr]\cos(c_nx) +(\rho_2\beta_n^2-k)B_nr_n+bB_n r_n'' ,
\end{split}\label{f4}\\
\rho_3 f_{5n}
&=
\bigl[\rho_3\mathbf{i}\beta_nD_n+{\underbrace{\kappa c_n^2D_n
-\mathbf{i}\gamma\beta_nc_nB_n}_{:=I_3}}\bigr]
\sin (c_nx)
-\mathbf{i}\gamma\beta_nB_n r_n' .\label{f5}
\end{align}

Now, we want to eliminate the leading-order terms in the coefficients of the trigonometric functions. For this, we choose $\beta_n$ such that $I_1=0$, $A_n$ such that $I_2=0$ and $D_n$ such that $I_3=0$ (the last two in terms of $B_n$, which will be chosen later). That is, we take
\begin{equation}\label{1309267}
\beta_n=c_n\sqrt{\frac{k}{\rho_1}},\qquad   A_n=\left(\frac{\rho_2}{\rho_1}-\frac{b}{k}\right)c_nB_n,\qquad D_n =\underbrace{\frac{\mathbf{i}\gamma}{\kappa}\sqrt{\frac{k}{\rho_1}}}_{:=\eta}B_n.
\end{equation}

Then, \eqref{f2}, \eqref{f4} and \eqref{f5} become
$$
\begin{aligned}
\rho_1f_{2n}&=B_n\bigl[kc_n
\sin (c_nx)
+k r_n' \bigr],\\
\rho_2f_{4n}&=B_n\bigl[(k+\gamma \eta c_n)\cos(c_nx)+(\rho_2\beta_n^2-k)r_n+b r_n'' \bigr],\\
\rho_3 f_{5n} & = B_n\mathbf{i}\bigl[\rho_3\eta\beta_n\sin (c_nx) -\gamma\beta_n r_n'\bigr].
\end{aligned}
$$

Finally, we choose $B_n$ so that $(f_{2n})$, $(f_{4n})$ and $(f_{5n})$ remain bounded in $L^2$. Since the last equalities combined with \eqref{1709262}-\eqref{1709264} imply that
\begin{equation}\label{1309264}
\|f_{2n}\|_2^2=|B_n|^2O(n^2),\quad \|f_{4n}\|_2^2=|B_n|^2O(n^{3})\quad \text{and}\quad \|f_{5n}\|_2^2=|B_n|^2O(n^{3}),
\end{equation}
it is enough to choose $B_n$ such that $|B_n|^2=O(n^{-3})$. For this, we take
\begin{equation}\label{1309265}
B_n=\frac{1}{n^{3/2}}.
\end{equation}
From \eqref{1309266} and \eqref{1309267}, it is clear that $\beta_n\to\infty$ and thus the first condition in \eqref{1209261} is satisfied. From \eqref{f1}, \eqref{f3}, \eqref{1309264}, \eqref{1309265} and \eqref{1709265}, we conclude that the second condition in \eqref{1209261} is satisfied. In order to see that the third condition in \eqref{1209261} is also satisfied, note that
$$
\begin{aligned}
&\|U_n\|_{\mathcal{H}}^2\overset{\eqref{norm}}{\geq}\rho_1\|\Phi_n\|_2^2\overset{\eqref{1309263}}{=}\rho_1\beta_n^2|A_n|^2
\int_0^\ell\sin ^2(c_nx)\,dx= {\underbrace{\left(\frac{\rho_2}{\rho_1}-\frac{b}{k}\right)}_{\text{$\neq 0$ by \eqref{nEWSC}}}}^2\frac{k\pi^4}{2\ell^3}n
\to\infty.
\end{aligned}
$$
\end{proof}

\section{Bresse system}

\subsection{Background}

Let $L^2$, $H^2$ and $H_0^1$ be as before. Write $U=(\varphi,\Phi,\psi,\Psi,w,W,\theta)$, $\tilde{U}=(\tilde{\varphi},\tilde{\Phi},\tilde{\psi},\tilde{\Psi},\tilde{w},\tilde{W},\tilde{\theta})$ and define
$$\mathcal{H}=[H_0^1\times L^2]^3\times L^2
$$
with norm
\begin{equation}\label{normb}
    \begin{aligned}
\|U\|_{\mathcal {H}}^2=&\;\rho_1\|\Phi\|_{2}^2+\rho_2\|\Psi\|_{2}^2+\rho_1\|W\|_{2}^2+k\|\varphi_x+\psi+lw\|_{2}^2\\
&+b\|\psi_x\|_{2}^2+k_0\|w_x-l\varphi\|_2^2+{\rho_3}\|\theta\|_{2}^2
\end{aligned}
\end{equation}
induced by the inner product
$$\begin{aligned}
(\tilde{U},U)_{\mathcal {H}}=&\;\rho_1(\tilde{\Phi},\Phi)_{2}+\rho_2(\tilde{\Psi},\Psi)_{2}+\rho_1(\tilde{W},W)_{2}+k(\tilde{\varphi}_x+\tilde{\psi}+l\tilde{w},\varphi_x+\psi+lw)_{2}\\
&+b(\tilde{\psi}_x,\psi_x)_{2}+k_0(\tilde{w}_x-l\tilde{\varphi},w_x-l\varphi)_2+{\rho_3}(\tilde{\theta},\theta)_{2}.
\end{aligned}$$

Now, define $\mathcal{A}:D(\mathcal{A})\to \mathcal{H}$
by
$$\mathcal{A}U=\begin{bmatrix}
\Phi\\
\frac{1}{\rho_1}[k(\varphi_x+\psi+lw)_x+k_0l(w_x-l\varphi)]\\
\Psi\\
\frac{1}{\rho_2}[b\psi_{xx}-k(\varphi_x+\psi+lw)-\gamma\theta_x]\\
W\\
\frac{1}{\rho_1}[k_0(w_x-l\varphi)_x-kl(\varphi_x+\psi+lw)]\\
\frac{1}{\rho_3}[\kappa\theta_{xx}-\gamma\Psi_x]
\end{bmatrix}$$
with domain
$$ D(\mathcal{A})=[(H^2\cap H_0^1)\times H_0^1]^3\times (H^2\cap H^1_0).
$$

In terms of this operator, system \eqref{sistemab} with boundary conditions \eqref{bcdb} and initial conditions
$$
\varphi=\varphi_0,\quad \varphi_t=\varphi_1,\quad \psi=\psi_0,\quad \psi_t=\psi_1,\quad w=w_0,\quad w_t=w_1,\quad \theta=\theta_0 
\qquad\text{on}\qquad\{0\}\times (0,\ell)
$$
can be written as
\begin{equation}\label{PACb}\left\{\begin{aligned}
&U_t(t)=\mathcal{A} U(t),\quad t> 0\\
&U(0)=U_0,
\end{aligned}\right.\end{equation}
where $U(t)=(\varphi(t),\allowbreak\varphi_t(t),\allowbreak\psi(t),\allowbreak\psi_t(t),\allowbreak w(t),\allowbreak w_t(t),\allowbreak \theta(t))$ and $U_0=(\varphi_0,\allowbreak \varphi_1,\allowbreak \psi_0,\allowbreak \psi_1,\allowbreak w_0,\allowbreak w_1,\allowbreak \theta_0)$.

\begin{theorem}[see \cite{Fatori10}] Let $\mathcal{A}$ be the operator defined above. Then:
\begin{enumerate}[(a)]
    \item $\Re(\mathcal{A}U,U)_{\mathcal{H}}=-\kappa\|\theta_x\|_2^2$ for all $U\in D(\mathcal{A})$.
    \item $\mathbf{i}\mathbb R\subset \rho(\mathcal{A})$.
    \item $\mathcal{A}$ is the infinitesimal generator of a $C_0$ semigroup of contractions on $\mathcal{H}$ and thus \eqref{PACb} has a unique classical solution for each $U_0\in D(\mathcal{A})$, and a unique mild solution for each $U_0\in \mathcal{H}$.
    \item\label{db} If condition \eqref{EWSCb} is satisfied, then the semigroup generated by $\mathcal{A}$ is exponentially stable.
\end{enumerate}
\end{theorem}

\subsection{Lack of exponential stability}

In this section, we prove the converse of item \eqref{db}, obtaining a complete characterization of the exponential stability in terms of the coefficients of the system: system \eqref{sistemab} with Dirichlet boundary conditions \eqref{bcdb} is exponentially stable if and only if ${\rho_1}/k={\rho_2}/{b}$ and $k=k_0$. To the best of our knowledge, this result was previously known only for mixed boundary conditions, while the case of full Dirichlet boundary conditions has usually remained open in the context of Bresse systems.

\begin{theorem}\label{mainthmb}
 Let $\mathcal{A}$ be the operator defined in the previous section. If
\begin{equation}\label{unequalb}
\frac{\rho_1}{k}\neq \frac{\rho_2}{b}
\qquad\text{or}\qquad
k\neq k_0,
\end{equation}
then the semigroup generated by $\mathcal{A}$ is not exponentially stable.
\end{theorem}

\begin{proof} As in the proof of Theorem \ref{mainthm}, it is enough to show that there exist sequences $(\beta_n)\subset\mathbb{R}$ and $(U_n)\subset D(\mathcal{A})$ satisfying \eqref{1209261}. In order to construct such sequences, let $c_n$ and $r_n$ be exactly as in the proof of Theorem~\ref{mainthm}. A natural attempt to extend the previous approach to the case of the Bresse system would be to take \eqref{1309263} together with $w_{n}(x)
=C_n\bigl[\cos(c_nx)-r_n(x)\bigr]$ and $W_{n}=\mathbf{i}\beta_n C_n[\cos(c_nx)-r_n(x)]$. This approach works well in the case ${\rho_1}/{k}\neq {\rho_2}/{b}$ and $k\neq k_0$, but is not good enough to cover the other cases included in \eqref{unequalb}. Therefore, it is convenient to consider different sequences in separate cases.

\vspace{\baselineskip}

\textbf{Case 1:} Suppose that
\begin{equation}\label{case1}
k\neq k_0.    
\end{equation}
Let us take $\beta_n\in\mathbb R$ and $A_n,B_n\in\mathbb C$ to be chosen later and define $U_n=
(\varphi_n,\allowbreak\Phi_n,\allowbreak\psi_n,\allowbreak\Psi_n,\allowbreak w_n,\allowbreak W_n,\allowbreak\theta_n)$, where
\begin{equation}\label{1309263b}
\begin{array}{llll}
\varphi_{n}(x)=A_n[\cos(c_n x)-r_n], &\psi_{n}(x)
=0, &w_n=B_n\sin(c_nx),\\
\Phi_{n}(x)
=\mathbf{i}\beta_nA_n[\cos(c_n x)-r_n], &
\Psi_{n}(x)
=0,&W_n=\mathbf{i}\beta_nB_n\sin(c_nx),\\
\theta_{n}(x)
=0.&
\end{array}
\end{equation}
Due to \eqref{1209262}, it is clear that $\varphi_n,\Phi_n,\psi_n,\Psi_n,w_n,W_n,\theta_n\in H_0^1$ and thus $U_n\in D(\mathcal{A})$. Let us write
\begin{equation}\label{1309262b}
(\mathbf{i}\beta_n I-\mathcal{A})U_n=(f_{1n},f_{2n}, f_{3n},f_{4n},f_{5n},f_{6n},f_{7n}).    
\end{equation}
Then,
\begin{equation}\label{1709265b}
\begin{aligned}
\|(\mathbf{i}\beta_n I-\mathcal{A})U_n\|^2_{\mathcal{H}}\overset{\eqref{normb}}{=}&\;
\rho_1\|f_{2n}\|_{2}^2+\rho_2\|f_{4n}\|_{2}^2+\rho_1\|f_{6n}\|_{2}^2+k\|f_{1n,x}+f_{3n}+lf_{5n}\|_{2}^2\\
&+b\|f_{3n,x}\|_{2}^2+k_0\|f_{5n,x}-lf_{1n}\|_2^2 +{\rho_3}\|f_{7n}\|_{2}^2.
\end{aligned}
\end{equation}
In terms of its components, \eqref{1309262b} can be written as
\begin{equation}\label{1709261b}
\begin{aligned}
\mathbf{i}\beta_n\varphi_{n}-\Phi_{n}
&=f_{1n},\\
\rho_1\mathbf{i}\beta_n\Phi_{n}
-k(\varphi_{n,x}+\psi_n+lw_n)_x
-k_0l(w_{n,x}-l\varphi_n)
&=\rho_1 f_{2n},\\
\mathbf{i}\beta_n\psi_{n}-\Psi_{n}
&=f_{3n},\\
\rho_2\mathbf{i}\beta_n\Psi_{n}
-b\psi_{n,xx}
+k(\varphi_{n,x}+\psi_n+lw_n)
+\gamma\theta_{n,x}
&=\rho_2 f_{4n},\\
\mathbf{i}\beta_n w_{n}-W_{n}
&=f_{5n},\\
\rho_1\mathbf{i}\beta_n W_{n}
-k_0(w_{n,x}-l\varphi_n)_x
+kl(\varphi_{n,x}+\psi_n+lw_n)
&=\rho_1 f_{6n},\\
\rho_3\mathbf{i}\beta_n\theta_{n}
-\kappa\theta_{n,xx}
+\gamma\Psi_{n,x}
&=\rho_3 f_{7n}.
\end{aligned}
\end{equation}

Substituting \eqref{1309263b} in \eqref{1709261b}, we obtain
\begin{align}
f_{1n}&=0,\label{f1b}\\
\rho_1 f_{2n}
&=
\bigl[
k_0l^2A_n+{\underbrace{(k c_n^2-\beta_n^2\rho_1)A_n
-(k+k_0)l c_nB_n}_{:=I_1}}
\bigr]\cos(c_nx)+(\rho_1\beta_n^2-k_0l^2)A_n r_n+kA_n r_n'',\label{f2b}\\
f_{3n}&=0,\label{f3b}\\
\rho_2 f_{4n}
&=
k(-c_nA_n+lB_n)\sin(c_nx)
-kA_n r_n',\label{f4b}\\
f_{5n}&=0,\label{f5b}\\
\rho_1 f_{6n}
&=
\bigl[
-(k+k_0)l c_nA_n+{\underbrace{(k_0c_n^2-\beta_n^2\rho_1)B_n}_{:=I_2}}
+kl^2B_n
\bigr]\sin(c_nx)
-(k+k_0)lA_n r_n',\label{f6b}\\
\rho_3 f_{7n}&=0.\label{f7b}
\end{align}

Now, in order to eliminate the leading-order terms in the coefficients of the trigonometric functions, we choose $\beta_n$ such that $I_2=0$ and $B_n$ such that $I_1=0$ (the last one in terms of $A_n$, which is left to be chosen later). That is, we take
\begin{equation}\label{1309267b}
\beta_n=c_n\sqrt{\frac{k_0}{\rho_1}},\qquad   B_n={\underbrace{(k-k_0)\frac{1}{(k+k_0)l}}_{:=\eta}}c_nA_n.
\end{equation}

Then, \eqref{f2b}, \eqref{f4b} and \eqref{f6b} become
\begin{align*}
\rho_1 f_{2n}
&=A_n\bigl[k_0l^2\cos(c_nx)
+(\rho_1\beta_n^2-k_0l^2)r_n
+kr_n''\bigr],\\
\rho_2 f_{4n}
&=A_n\bigl[k(l\eta-1)c_n\sin(c_nx)
-k r_n'\bigr],\\
\rho_1 f_{6n}
&=A_n\bigl[(kl\eta-k-k_0)lc_n\sin(c_nx)
-(k+k_0)l r_n'\bigr].
\end{align*}

Finally, we choose $A_n$ so that $(f_{2n})$, $(f_{4n})$ and $(f_{6n})$ remain bounded in $L^2$. Since the last equalities combined with \eqref{1709262}-\eqref{1709264} imply that
\begin{equation}\label{1309264b}
\|f_{2n}\|_2^2=|A_n|^2O(n^3),\quad \|f_{4n}\|_2^2=|A_n|^2O(n^{2})\quad \text{and}\quad \|f_{6n}\|_2^2=|A_n|^2O(n^{2}),
\end{equation}
it is enough to choose $A_n$ such that $|A_n|^2=O(n^{-3})$. For this, we take
\begin{equation}\label{1309265b}
A_n=\frac{1}{n^{3/2}}.
\end{equation}
From \eqref{1309266} and \eqref{1309267b}, it is clear that $\beta_n\to\infty$ and thus the first condition in \eqref{1209261} is satisfied. From \eqref{f1b}, \eqref{f3b}, \eqref{f5b}, \eqref{f7b}, \eqref{1309264b}, \eqref{1309265b} and \eqref{1709265b}, we conclude that the second condition in \eqref{1209261} is satisfied. In order to see that the third condition in \eqref{1209261} is also satisfied, note that
$$
\begin{aligned}
&\|U_n\|_{\mathcal{H}}^2\overset{\eqref{normb}}{\geq}\rho_1\|W_n\|_2^2\overset{\eqref{1309263b}}{=}\rho_1\beta_n^2|B_n|^2
\int_0^\ell\sin ^2(c_nx)\,dx={\underbrace{(k-k_0)}_{\text{$\neq 0$ by \eqref{case1}}}}^2\frac{k_0\pi^4}{2\ell^3(k+k_0)^2l^2}n
\to\infty.
\end{aligned}
$$

\vspace{\baselineskip}

\textbf{Case 2:} Suppose that
\begin{equation}\label{case2}
\frac{\rho_1}{k}\neq \frac{\rho_2}{b}
\qquad\text{and}\qquad
k=k_0.
\end{equation}
Let us take $\beta_n,\omega\in\mathbb R$ and $A_n,B_n,C_n,D_n\in\mathbb C$ to be chosen later, define $E(x)=e^{\mathbf{i}\omega x}$ and $U_n=
(\varphi_n,\allowbreak\Phi_n,\allowbreak\psi_n,\allowbreak\Psi_n,\allowbreak w_n,\allowbreak W_n,\allowbreak\theta_n)$, where
\begin{equation}\label{1309263b2}
\begin{array}{llll}
\varphi_{n}(x)=A_nE\sin(c_nx), &\psi_{n}(x)
=B_nE[\cos(c_nx)-r_n(x)], &w_n=C_n E \sin(c_nx),\\
\Phi_{n}(x)
=\mathbf{i}\beta_nA_nE\sin(c_nx), &
\Psi_{n}(x)
=\mathbf{i}\beta_n B_nE[\cos(c_nx)-r_n(x)],&W_n=\mathbf{i}\beta_nC_n E \sin(c_nx),\\
\theta_{n}(x)
=D_nE\sin(c_nx).&
\end{array}
\end{equation}
Due to \eqref{1209262}, it is clear that $\varphi_n,\Phi_n,\psi_n,\Psi_n,w_n,W_n,\theta_n\in H_0^1$ and thus $U_n\in D(\mathcal{A})$. Let us use the notation \eqref{1309262b} again.
Then, \eqref{1709265b} and \eqref{1709261b} are also valid. Substituting \eqref{1309263b2} in \eqref{1709261b}, we obtain
\begin{equation}\label{f1b2}
    f_{1n}=f_{3n}=f_{5n}=0,
\end{equation}

\begin{equation}\label{f2b2}
    \begin{aligned}
        \rho_1f_{2n}
=&\,
\bigl[
({\underbrace{kc_n^2-\rho_1\beta_n^2}_{:=I_1}}+k\omega^2
+k_0l^2)A_n
+kc_nB_n-l(k+k_0)\omega\mathbf{i} C_n
\bigr]E\sin(c_nx)\\
&+
\bigl[
-2k\omega c_nA_n
- k\omega B_n+l(k+k_0)\mathbf{i} c_n C_n
\bigr]\mathbf{i}E\cos(c_nx)
+\bigl[r_n'
+\mathbf i \omega r_n\bigr]kEB_n,
    \end{aligned}
\end{equation}

\begin{equation}\label{f4b2}
    \begin{aligned}
\rho_2f_{4n}
=&\,
\bigl[
\omega k A_n
+2b\omega c_nB_n
-lk\mathbf{i} C_n +\gamma\omega D_n
\bigr]\mathbf{i}E\sin(c_nx)\\
&+
\bigl[{\underbrace{k c_n A_n+
(bc_n^2-\rho_2\beta_n^2)B_n}_{:=I_2}}+(b\omega^2
+k
)B_n
+\gamma c_n D_n
\bigr]E\cos(c_nx)\\
&+
\bigl[
(\rho_2\beta_n^2
-b\omega^2
-k)r_n
+
2\mathbf i b\omega r_n'
+
br_n''
\bigr]E B_n,
    \end{aligned}
\end{equation}

\begin{equation}\label{f6b2}
    \begin{aligned}
    \rho_1f_{6n}
=&\,
\bigl[
l(k+k_0)\mathbf{i}\omega A_n+({\underbrace{
k_0c_n^2-\rho_1\beta_n^2}_{:=I_3}}+k_0\omega^2
+kl^2)C_n
\bigr] E \sin(c_nx)\\
&+
\bigl[
l(k+k_0)c_nA_n
+klB_n-2k_0\mathbf{i}\omega c_n C_n
\bigr]E \cos(c_nx)-kl E B_n r_n,
    \end{aligned}
\end{equation}

\begin{equation}\label{f7b2}
\begin{aligned}
\rho_3f_{7n}
=&\,
\bigl[
{\underbrace{-\mathbf{i}\gamma\beta_nc_nB_n+\kappa c_n^2 D_n}_{:=I_4}}+(\kappa \omega^2+\mathbf i\rho_3\beta_n)D_n
\bigr]E\sin(c_nx)\\
&-
\bigl[
\gamma\omega\beta_n B_n+2\mathbf i\kappa\omega c_nD_n
\bigr]E\cos(c_nx)+
\bigl[\gamma\omega\beta_n r_n
-\mathbf i\gamma\beta_n r_n'\bigr] E B_n.
\end{aligned}    
\end{equation}

Now, we want to eliminate the leading-order terms in the coefficients of the trigonometric functions. In order to obtain $I_1=0$, we choose
\begin{equation}\label{1110261}
\beta_n=c_n\sqrt{\frac{k}{\rho_1}}.
\end{equation}
Due to the equality in \eqref{case2}, this choice of $\beta_n$ also implies that $I_3=0$. In order to obtain $I_2=0$ and $I_4=0$, we choose
$$A_n={\underbrace{\left(\frac{k_0}{\rho_1}
-\frac{b}{\rho_2}\right)\frac{\rho_2}{k}}_{:=\eta_1}}c_nB_n,\qquad D_n={\underbrace{\mathbf{i}\frac{\gamma}{\kappa} \sqrt{\frac{k_0}{\rho_1}}}_{:=\eta_2}}B_n.$$

Then, \eqref{f2b2}-\eqref{f7b2} become
\begin{equation}\label{f2b22}
    \begin{aligned}
        \rho_1f_{2n}
=&\,
\bigl[
((k\omega^2
+k_0l^2)\eta_1
+k)c_nB_n-l(k+k_0)\omega\mathbf{i} C_n
\bigr]E\sin(c_nx)\\
&+
\bigl[
- k\omega B_n{\underbrace{-2k\omega \eta_1 c_n^2 B_n
+l(k+k_0)\mathbf{i} c_n C_n}_{:=J_1}}
\bigr]\mathbf{i}E\cos(c_nx)
+\bigl[r_n'
+\mathbf i \omega r_n\bigr]kEB_n,
    \end{aligned}
\end{equation}

\begin{equation}\label{f4b22}
    \begin{aligned}
\rho_2f_{4n}
=&\,
\bigl[
(\omega k \eta_1
+2b\omega)c_nB_n
-lk\mathbf{i} C_n +\gamma\omega \eta_2 B_n
\bigr]\mathbf{i}E\sin(c_nx)\\
&+
\bigl[(b\omega^2
+k
)B_n
+\gamma \eta_2 c_n B_n
\bigr]E\cos(c_nx)\\
&+
\bigl[
(\rho_2\beta_n^2
-b\omega^2
-k)r_n
+
2\mathbf i b\omega r_n'
+
br_n''
\bigr]E B_n,
    \end{aligned}
\end{equation}

\begin{equation}\label{f6b22}
    \begin{aligned}
    \rho_1f_{6n}
=&\,
\bigl[
l(k+k_0)\mathbf{i}\omega \eta_1 c_n B_n+(k_0\omega^2
+kl^2)C_n
\bigr] E \sin(c_nx)\\
&+
\bigl[klB_n+
{\underbrace{l(k+k_0)\eta_1 c_n^2 B_n
-2k_0\mathbf{i}\omega c_n C_n}_{:=J_2}}
\bigr]E \cos(c_nx)-kl E B_n r_n,
    \end{aligned}
\end{equation}

\begin{equation}\label{f7b22}
\begin{aligned}
\rho_3f_{7n}
=&\,
(\kappa \omega^2+\mathbf i\rho_3\beta_n)\eta_2 E B_n\sin(c_nx)\\
&-
\bigl[
\gamma\omega\beta_n+2\mathbf i\kappa\omega \eta_2 c_n
\bigr]E B_n\cos(c_nx)+
\bigl[\gamma\omega\beta_n r_n
-\mathbf i\gamma\beta_n r_n'\bigr] E B_n.
\end{aligned}    
\end{equation}

Note that the previous choices generated new higher-order terms in $J_1$ and $J_2$. In order to eliminate them, we first choose $C_n$ such that $J_1=0$, that is, we take
$$C_n=\underbrace{\frac{2k\omega \eta_1}{l(k+k_0)\mathbf{i}}}_{:=\eta_3} c_nB_n.$$
Then, $J_2$ becomes $\left[l(k+k_0) -\omega^2 \frac{4kk_0 }{l(k+k_0)} \right]\eta_1c_n^2B_n$. Therefore, in order to eliminate $J_2$, we choose
$$\omega =\frac{l(k+k_0)}{2\sqrt{kk_0}}.$$

Due to these choices, \eqref{f2b22}-\eqref{f6b22} become
\begin{equation}\label{f2b23}
    \begin{aligned}
        \rho_1f_{2n}
=&\,
\bigl[
(k\omega^2
+k_0l^2)\eta_1
+k-l(k+k_0)\omega\mathbf{i} \eta_3 
\bigr] c_n E B_n\sin(c_nx)
- k\omega B_n \mathbf{i}E\cos(c_nx)\\
&+\bigl[r_n'
+\mathbf i \omega r_n\bigr]kEB_n,
    \end{aligned}
\end{equation}

\begin{equation}\label{f4b23}
    \begin{aligned}
\rho_2f_{4n}
=&\,
\bigl[
(\omega k \eta_1
+2b\omega-lk\mathbf{i} \eta_3 )c_n
 +\gamma\omega \eta_2
\bigr] \mathbf{i}E B_n\sin(c_nx)
+
\bigl[(b\omega^2
+k
)
+\gamma \eta_2 c_n 
\bigr] B_nE\cos(c_nx)\\
&+
\bigl[
(\rho_2\beta_n^2
-b\omega^2
-k)r_n
+
2\mathbf i b\omega r_n'
+
br_n''
\bigr]E B_n,
    \end{aligned}
\end{equation}

\begin{equation}\label{f6b23}
    \begin{aligned}
    \rho_1f_{6n}
=&\,
\bigl[
l(k+k_0)\mathbf{i}\omega \eta_1 +(k_0\omega^2
+kl^2)\eta_3
\bigr]c_nE B_n \sin(c_nx)+klE B_n \cos(c_nx)-kl E B_n r_n.
    \end{aligned}
\end{equation}

Finally, we choose $B_n$ so that $(f_{2n})$, $(f_{4n})$, $(f_{6n})$ and $(f_{7n})$ remain bounded in $L^2$. Note that \eqref{f7b22}-\eqref{f6b23} combined with \eqref{1709262}-\eqref{1709264} imply that
\begin{equation}\label{1309264b2}
\|f_{2n}\|_2^2=|B_n|^2O(n^2),\quad \|f_{4n}\|_2^2=|B_n|^2O(n^{3}),\quad \|f_{6n}\|_2^2=|B_n|^2O(n^{2}),\quad \|f_{7n}\|_2^2=|B_n|^2O(n^{3}).
\end{equation}
Therefore, it is enough to choose $B_n$ such that $|B_n|^2=O(n^{-3})$. For this, we take
\begin{equation}\label{1309265b2}
B_n=\frac{1}{n^{3/2}}.
\end{equation}
From \eqref{1309266} and \eqref{1110261}, it is clear that $\beta_n\to\infty$ and thus the first condition in \eqref{1209261} is satisfied. From \eqref{f1b2}, \eqref{1309264b2}, \eqref{1309265b2} and \eqref{1709265b}, we conclude that the second condition in \eqref{1209261} is satisfied. In order to see that the third condition in \eqref{1209261} is also satisfied, note that
$$
\begin{aligned}
&\|U_n\|_{\mathcal{H}}^2\overset{\eqref{normb}}{\geq}\rho_1\|\Phi_n\|_2^2\overset{\eqref{1309263b2}}{= }\rho_1\beta_n^2|A_n|^2
\int_0^\ell |e^{ \mathbf{i} \omega x}|^2\sin ^2(c_nx)\,dx={\underbrace{\left(\frac{k_0}{\rho_1}
-\frac{b}{\rho_2}\right)}_{\text{$\neq 0$ by \eqref{case2}}}}^2\frac{\pi^4\rho_2^2}{2k\ell^3}n
\to\infty.
\end{aligned}
$$
\end{proof}

\bibliographystyle{plain}

\bibliography{cas-refs}

\end{document}